\documentclass[10pt,reqno]{amsart}
\usepackage{amscd,amsmath,amsopn,amssymb,amsthm,multicol}
\usepackage{hyperref}
\DeclareMathOperator{\Id}{Id}

\DeclareMathOperator{\diag}{diag}

\DeclareMathOperator{\Ad}{Ad}

\DeclareMathOperator{\trace}{trace}

\DeclareMathOperator{\Pol}{Pol}

\renewenvironment{proof}[1][Proof]{\textbf{#1.} }
{\ \rule{0.5em}{0.5em}}

\newtheorem{theorem}{Theorem}
\newtheorem{prop}{Proposition}
\newtheorem{lemma}{Lemma}

\newtheorem{quest}{Question}

\theoremstyle{definition}

\newtheorem{remark}{Remark}

\begin{document}

\title
[The classification \dots]
{The classification of $\delta$-homogeneous Riemannian manifolds with positive Euler characteristic}
\author{V.N.~Berestovski\u\i, E.V.~Nikitenko, Yu.G.~Nikonorov}

\address{Berestovski\u\i\  Valeri\u\i\  Nikolaevich \newline
Omsk Branch of Sobolev Institute of Mathematics SD RAS \newline
644099, Omsk, ul. Pevtsova,
13,Russia}
\email{berestov@ofim.oscsbras.ru}

\address{Nikitenko Evgeni\u\i\  Vitalievich\newline
Rubtsovsk Industrial Institute \newline of Altai State Technical
University after I.I.~Polzunov \newline 658207, Rubtsovsk,
ul. Traktornaya, 2/6, Russia}
\email{nikit@inst.rubtsovsk.ru}

\address{Nikonorov\ Yuri\u\i\  Gennadievich\newline
Rubtsovsk Industrial Institute \newline of Altai State Technical
University after I.I.~Polzunov \newline 658207, Rubtsovsk,
ul. Traktornaya, 2/6, Russia}
\email{nik@inst.rubtsovsk.ru}

\begin{abstract}

The authors give a short survey of previous results on $\delta$-homogeneous Riemannian
manifolds, forming a new proper subclass of geodesic orbit spaces with non-negative
sectional curvature, which properly includes the class of all normal homogeneous
Riemannian manifolds. As a continuation and an application of these results, they prove that the family of
all compact simply connected indecomposable $\delta$-homogeneous Riemannian manifolds with positive Euler
characteristic, which are not normal homogeneous, consists exactly of all generalized flag manifolds
$Sp(l)/U(1)\cdot Sp(l-1)=\mathbb{C}P^{2l-1}$, $l\geq 2$,
supplied with invariant Riemannian metrics of positive sectional curvature with
the pinching constants (the ratio of the minimal sectional curvature to the maximal one)
in the open interval $(1/16, 1/4)$. This implies very unusual geometric properties of the adjoint representation
of $Sp(l)$, $l\geq 2$. Some unsolved questions are suggested.

\vspace{2mm}

\vspace{2mm} \noindent Key words and phrases: homogeneous spaces,
homogeneous spaces of positive Euler characteristic, g.o. spaces, normal
homogeneous Riemannian manifolds, $\delta$-homogeneous Riemannian manifolds,
Clifford-Wolf translations, Clifford-Wolf homogeneous Riemannian manifolds, geodesics.
\end{abstract}

\maketitle

\section{Introduction}

In this paper we finish the classification of compact simply connected
indecomposable $\delta$-homogeneous, but not normal homogeneous, Riemannian manifolds with positive Euler
characteristic (see the previous papers \cite{BerPl}, \cite{BerNik}, \cite{BerNik1}, \cite{BerNik2}, \cite{BerNik3}).
Thus it is appropriate to give
in this introduction a short informal survey of results obtained for $\delta$-homogeneous spaces and indicate some
unsolved questions.

Let us begin with a short description of well-known classes of
Riemannian homogeneous manifolds closely related to the object of this paper.

The Riemannian symmetric spaces introduced and classified by E.~Cartan in \cite{Ca} are the best-studied.
The Riemannian symmetric spaces form a proper subclass in the class of naturally reductive homogeneous
Riemannian manifolds defined by K.~Nomizu \cite{KN},
and in the class of weakly symmetric Riemannian manifolds introduced by A.~Selberg
\cite{S}. Any Riemannian symmetric space of nonnegative sectional curvature is a normal homogeneous Riemannian
manifold in sense of M.~Berger \cite{Berg}. Any normal homogeneous Riemannian manifold is naturally reductive.
Finally, all Riemannian manifolds listed above are geodesic orbit (g.o.) spaces. The latter spaces have been defined
and studied at the first time by O.~Kowalski and L.~Vanhecke in the paper \cite{KV}. The
assertion that weakly symmetric Riemannian manifolds are g.o. spaces have been proved in the paper \cite{BKV}.
A.~Selberg proved in \cite{S} that any weakly symmetric space is a commutative space. It follows from \cite{AV},
that in the compact case, any commutative space is weakly symmetric. On the other hand, this is not true in
general \cite{La1}.
It is interesting that a smooth connected Riemannian manifold is homogeneous if the
isometrically invariant differential operators on the manifold form a commutative
algebra \cite{Mao}.

A Riemannian manifold $(M,\mu)$ is \textit{symmetric} if for any its point $x$
there is an isometry $f$ of $(M,\mu)$ such that $f(\gamma(t))=\gamma(-t)$ for all (arc-length parameterized)
geodesics $\gamma(t)$, $t\in \mathbb{R}$, with $\gamma(0)=x$.
A manifold $(M,\mu)$ is \textit{weakly symmetric} if for any geodesic $\gamma(t)$, $t\in \mathbb{R}$,
there is an isometry $f$ such that $f(\gamma(t))=\gamma(-t)$ (for all $t\in \mathbb{R}$).
$(M,\mu)$ is \textit{commutative} if it admits a transitive isometry Lie group $G$ such that the algebra of
$G$-invariant differential operators is commutative (see also \cite{W1}). $(M,\mu)$ is \textit{normal homogeneous}
(respectively \textit{naturally reductive homogeneous}) if there is a transitive isometry Lie group of $(M,\mu)$ with
the stabilizer subgroup $H\subset G$ at some point $x\in M$ and a bi-invariant (non-degenerate) Riemannian
(respectively, semi-Riemannian) metric tensor $\nu$ on $G$ such that the natural projection
$p:(G,\nu)\rightarrow (M=G/H,\mu)$ is a (semi-)Riemannian submersion \cite{FIP}. The latter means that for any
element $g\in G$, the differential $dp(g)$ maps isometrically $(\ker dp(g))^\perp$ onto the tangent Euclidean
space $(M_{p(g)},\mu(p(g)))$. If this condition is satisfied for a particular Lie group $G$, we say also that
$(M=G/H,\mu)$ is $G$-normal homogeneous (respectively, $G$-naturally reductive). A Riemannian manifold $(M,\mu)$
is \textit{geodesic orbit}
(g.o.) if any its geodesic is an orbit of some one-parameter group of isometries.
All the Riemannian manifolds above are homogeneous.

In this paper we study $\delta$-homogeneous Riemannian manifolds, which can be
considered as a nearest metric generalization of normal homogeneous spaces.
Let us remark at first that, as a corollary of results in \cite{BerG}, the
above projection $p$ is a Riemannian submersion if and only if it is a \textit{submetry}. This means that, with
respect to the corresponding induced inner metrics $\rho_G$
and $\rho_M$, $p$ maps any closed ball $B_{\rho_G}(g,r), r\geq 0, g\in G,$ onto the closed ball $B_{\rho_M}(p(g),r)$.
Now we shall get a definition of a $\delta$-homogeneous
Riemannian manifold $(M,\mu)$ if in the above definition of a normal homogeneous manifold
we change the metric tensor $\nu$ to an inner metric $\rho_G$, not necessarily induced by the metric tensor $\nu$,
and the condition for $p$ to be a Riemannian submersion to the condition to be a submetry. If this condition is
satisfied for a particular Lie group $G$, we say also that $(M=G/H,\mu)$ is \textit{$G$-$\delta$-homogeneous}.
This discussion implies that every normal homogeneous Riemannian manifold is $\delta$-homogeneous.

Note that any bi-invariant inner metric $\rho_G$ on a Lie group $G$ is Finsler, i.e. it is induced by a
$\Ad(G)$-invariant norm $||\cdot||$ on the Lie algebra $\mathfrak{g}$ of $G$ \cite{Ber}. In fact, for the main,
compact case, in the above definition of $\delta$-homogeneity, one can take as $G$ the connected unit component
in the full isometry group $I(M)$ of $(M,\mu)$, and as $\rho_G$ the inner metric on
$G$, induced by the bi-invariant metric $d(g,h)=\max_{x\in M}\rho_M(g(x),h(x))$.
Then $\mathfrak{g}$ is naturally identified with the Lie algebra of Killing vector fields on $(M,\mu)$,
and the corresponding norm $||\cdot||$ will be the  \textit{Chebyshev norm}:
$||X||=\max_{x\in M}\sqrt{\mu(X(x),X(x))}$ \cite{BerNik, BerNik2}. The Chebyshev
norm is the minimal norm among all ($\Ad(G)$-invariant) norms on $\mathfrak{g}$,
satisfying the above definition of the $\delta$-homogeneity \cite{BerNik2}.

In fact, the above mentioned projection $p$ is a submetry if and only if its tangent linear map
$dp(e):(\mathfrak{g},||\cdot||)\rightarrow (M_{p(e)}, \mu(p(e)))$ of normed vector spaces is a submetry. In turn,
the last map is a submetry if and only if for every vector $v\in M_{p(e)}$ there is an element $X\in \mathfrak{g}$
such that
$X(p(e))=v$ and $||X||=\sqrt{\mu(X(p(e)),X(p(e)))}$ (this can be considered as a second
definition of the $\delta$-homogeneity). We refer to such vector field $X$ as a \textit{$\delta$-vector}. It necessarily
possess the property that its integral path through the point $p(e)$ is a geodesic in $(M,\mu)$, or in generally
accepted
terminology, it is a \textit{geodesic (g.o.) vector}. This implies that every $\delta$-homogeneous Riemannian
manifold is g.o. Also, this permits to apply the known properties of g.o.-vectors to study $\delta$-vectors.
Another simple but very useful fact is that for every compact homogeneous Riemannian manifold $(M,\mu)$, not
necessarily $\delta$-homogeneous, the $\Ad(G)$-orbit of any element in $\mathfrak{g}$ contains at least one
$\delta$-vector. Then, if $G$ is a matrix Lie group, hence $\mathfrak{g}$ is a matrix Lie algebra, one can use
the property that all elements of an $\Ad(G)$-orbit have one and the same characteristic polynomial. All the last
three properties are really used in the last section of \cite{BerNik} and Sections \ref{SO} and \ref{SP} of this paper
and give a very efficient method of the study.

Indeed, we can give a much simpler definition of $\delta$-homogeneity, which may be
applied to an arbitrary metric space: a metric space $(M,\rho_M)$ is $\delta$-homogeneous if for any two points $x$ and
$y$ from $M$, there exists an isometry $f$ of the
space $M$ onto itself which moves $x$ to $y$ and has the maximal displacement at the point $x$
(this means that $f(x)=y$ and $\rho_M(x,f(x))\geq \rho_M(z,f(z))$ for all points
$z\in M$) \cite{BerPl}. The equivalence of the above three definitions of the $\delta$-homogeneity for a
Riemannian manifold $(M,\mu)$ with the induced inner metric $\rho_M$ is proved in \cite{BerNik}.
Changing in the latter (metric) definition of the $\delta$-homogeneity
the inequality $\rho_M(x,f(x))\geq \rho_M(z,f(z))$ to the
equality $\rho_M(x,f(x))=\rho_M(z,f(z))$, we get a definition of a \textit{Clifford-Wolf homogeneous} metric space,
\cite{BerPl, BerNik4, BerNik5}.

Using the third, metric, definition of $\delta$-homogeneity, and methods of metric geometry, the authors of
\cite{BerPl} proved that every locally compact $\delta$-homogeneous inner (length) metric space with Aleksandrov
curvature bounded below has nonnegative Aleksandrov curvature. Such space is Riemannian $\delta$-homogeneous manifold
if and only if it
is finite-dimensional. Thus, as a corollary, we get that every $\delta$-homogeneous Riemannian manifold has
nonnegative sectional curvature. The last statement can be proved also by methods of Riemannian geometry.
Namely, to get this proof, one can simply combine the second definition
of $\delta$-homogeneous Riemannian manifold and Theorem 1 in \cite{BerNik7}, which implies that
$\mu(R(X(x),u)u,X(x))\geq 0$ for every nontrivial Killing vector field $X$ on a Riemannian manifold $(M,\mu)$,
attaining its maximal length at a point $x$, and every vector $u\in M_x$.

Now we state the main old and new results about $\delta$-homogeneous (particularly,
Clifford-Wolf homogeneous) Riemannian manifolds.

When we started to study $\delta$-homogeneous manifolds, having in
mind the above mentioned minimal property of the Chebyshev norm $||\cdot||$, we tried to prove the converse
statement by showing that for a $\delta$-homogeneous Riemannian
manifold $(G/H,\mu)$, the map $p:(G, \nu)\rightarrow (G/H,\mu)$ is a Riemannian submersion for a bi-invariant
Rimennian metric $\nu$, whose unit closed ball at $(G_e=\frak{g},\nu(e))$ is the Loewner-John ellipsoid for unit
closed ball $B$ at $(\frak{g},||\cdot||)$, i.e. the (unique) ellipsoid of maximal volume, inscribed into $B$.
But the last assertion has turned out to be false even for such normal homogeneous
Riemannian manifolds as spheres with non-constant positive sectional curvature \cite{BerNik2}.

Let us indicate some general properties of $\delta$-homogeneous Riemannian manifold,
which have been discussed in \cite{BerNik1, BerNik3} and proven in \cite{BerNik}.
Any such manifold has nonnegative sectional curvature and is a direct metric product
of an Euclidean space and compact indecomposable $\delta$-homogeneous Riemannian
manifolds (with possible omission of the mentioned factors). Conversely, any direct metric product of
$\delta$-homogeneous Riemannian manifolds is $\delta$-homogeneous.
Any locally isometric (particularly, universal) covering of every $\delta$-homogeneous
Riemannian manifold is itself $\delta$-homogeneous. All these assertions are true also for Clifford-Wolf
(CW-) homogeneous Riemannian manifolds. It follows from these results that the study of $\delta$- or CW-homogeneous
spaces mainly (even not
entirely) reduces to the case of indecomposable compact simply connected manifolds.

The main result of the papers \cite{BerNik, BerNik1} (which will be discussed
in more details later) is that the $\delta$-homogeneous Riemannian manifolds form a new subclass of g.o.
Riemannian manifolds, which contains the class of all normal homogeneous Riemannian manifolds, but does not
coincide with it. At the same time, the classes of
homogeneous spaces $G/H$ (with compact subgroup $H$) admitting invariant Riemannian metrics of any type: normal
homogeneous, $\delta$-homogeneous, or metrics with nonnegative sectional curvature, are all coincide \cite{BerNik1}.
The Riemannian space of any of these three types admits a transitive Lie group with compact Lie algebra, but the
full isometry group has compact Lie algebra only if the Euclidean factor has dimension no more than one.
Any space $G/H$ such that the Lie algebra of $G$ is compact, admits an invariant normal homogeneous metric,
 hence the metrics of other two types.

Another, very important result, which will be applied in this paper, states that any closed totally geodesic
submanifold of a $\delta$-homogeneous (respectively, g.o.) Riemannian manifold is itself $\delta$-homogeneous
(respectively, g.o.) \cite{BerNik}.

A simply connected (connected) Riemannian manifold is CW-homogeneous if and only if it
is isometric to a direct metric product of an Euclidean space, odd-dimensional spheres
of constant sectional curvature, and simply connected simple compact
Lie groups with bi-invariant Riemannian metrics (some of the factors may be missing)
\cite{BerNik4, BerNik5}. As a corollary, it is always symmetric and normal
homogeneous. Notice that as a main tool in the proof of this result were nontrivial Killing vector fields of
constant length, which have been studied also in \cite{BerNik6, BerNik7, BerNik8}. Any geodesic in CW-homogeneous
Riemannian manifold
is the integral path of some nontrivial Killing vector field of constant length \cite{BerNik4, BerNik5}. Thus,
in the compact case, such manifold has zero Euler characteristic.

Special attention has been and will be paid to the case of compact simply
connected $\delta$-homogeneous Riemannian manifolds with positive Euler characteristic.
By the B.~Kostant and Hopf-Samelson theorems \cite{On}, in more general, homogeneous case, every such
manifold is effective homogeneous space $M = G/H$ of a semisimple compact Lie group $G$, and each maximal torus
of the subgroup $H$ is a maximal torus of the group $G$, see \cite{On}. We proved that $M=G/H$ with an invariant
naturally reductive metric is normal homogeneous, and thus is $\delta$-homogeneous \cite{BerNik3}.
If $M$ is also indecomposable, then the Lie group $G$ is simple \cite{Kost}. One can find a complete
classification of such spaces in \cite{W}.

Using the existence of such classification and some results of A.L.~Onishchik \cite{On}
about full isometry groups of compact homogeneous Riemannian manifolds,
the authors  started a systematic search of all possible candidates for compact simply connected
indecomposable $\delta$-homogeneous Riemannian manifolds with positive Euler characteristic, which are not
normal homogeneous (really conjecturing at the same time that there are no such spaces).

The exclusion process in this search has had several stages. At first
all compact simple Lie groups whose roots have one and the same length
(as full connected Lie groups of motions for mentioned possible candidates) have been excluded, then all compact
simple exceptional Lie groups, and many homogeneous spaces of Lie groups $SO(2l+1)$ and $Sp(l)$.
These exclusions have been made after rather extensive and hard infinitesimal calculations on Lie algebras and work
with root systems and root decompositions. We finished the paper \cite{BerNik} by the list of possible candidates,
consisting only from the generalized flag manifolds $SO(2l+1)/U(l)$ and $Sp(l)/U(1)\cdot  Sp(l-1)=\mathbb{C}P^{2l-1}$
for $l \geq 2$.

Both families have many common properties. They start with the same space \linebreak
$SO(5)/U(2)=Sp(2)/U(1)\cdot  Sp(1)=\mathbb{C}P^3$, admit a two-parametric family of invariant Riemannian matrics,
all these metrics are g.o. and weakly symmetric, and mainly are not normal \cite{Zi96, Ta, AA}. Supplied with these
metrics, the spaces from both these families are total spaces of Riemannian submersions,
hence (nontrivial) fiber bundles, with irreducible symmetric Riemannian spaces with positive Euler characteristic as
bases and (totally geodesic) fibers, $SO(2l+1)/SO(2l)=S^{2l}$
and $SO(2l)/U(l)$ respectively for the first family and $Sp(l)/Sp(1)\cdot Sp(l-1)=\mathbb{H}P^{l-1}$ and
$Sp(1)/U(1)=S^2$ for the second family. The bases in all cases are two-point homogeneous. Note that the space
$SO(2l)/U(l)$ is usually treated as the set of complex structures on $\mathbb{R}^{2l}$ or the set of the
metric-compatible fibrations $S^1\rightarrow \mathbb{R}P^{2l-1}\rightarrow \mathbb{C}P^{l-1}$ \cite{Bes}, but one
can easily deduce from results in \cite{BerNik5} that it can be interpreted also as a connected component of
the set of all unit Killing vector fields on the round sphere $S^{2l-1}$. Also in \cite{BerNik} it have been stated
the same a priori constraints for parameters of spaces as possible invariant Riemannian $\delta$-homogeneous, but
not normal metrics, namely exactly strongly between the parameters of two distinct families of normal metrics.
Finally, at the end of the paper \cite{BerNik} the authors proved that a unique common member of these two families,
$SO(5)/U(2)=Sp(2)/U(1)\cdot  Sp(1)=\mathbb{C}P^3$, supplied with Riemannian invariant metrics with the mentioned a priori  parameters, is actually a $\delta$-homogeneous, but not normal homogeneous manifold. A quite different proof of this fact is given also in \cite{BerNik2}. All these metrics have positive sectional curvature, and using the paper \cite{Vol}, they can be characterized by the property, that their pinching constants
lie in the open interval $(1/2^4,1/2^2)$. The pinching constant $1/2^2$ corresponds to the famous Fubini-Studi
metric on $\mathbb{C}P^3$ and metrics, which are homothetic to it. All other normal homogeneous Riemannian metrics
on $\mathbb{C}P^3$ have the pinching constant $1/2^4$.

Notice that historically, $SO(5)/U(2)=Sp(2)/U(1)\cdot  Sp(1)=\mathbb{C}P^3$
was the first example of a compact non-naturally reductive homogeneous space admitting invariant g.o.
Riemannian metrics \cite{KV}. Maybe it is appropriate to notice
that the underlying manifold $\mathbb{C}P^3$ of the above homogeneous space is the
\textit{Penrose twistor space} \cite{Pen}, which can be interpreted, for example,
as the space of all compatible complex structures on the round $4$-dimensional sphere \cite{Bes}.

\begin{remark}
Nevertheless, there is one known essential distinction between the families $SO(2l+1)/U(l), l\geq 3,$ and
$Sp(l)/U(1)\cdot Sp(l-1)=\mathbb{C}P^{2l-1}, l \geq 2$: the spaces of the first family admit no invariant Riemannian
metrics of (strongly) positive sectional curvature, while all the spaces from the second family admit such metrics.
\end{remark}

Because of all these common properties, it was quite natural to conjecture that all other spaces from both families
admit invariant $\delta$-homogeneous but not normal homogeneous Riemannian metric. Surprisingly enough,  it turns out,
that with respect to this property, they behave itself quite differently. In this paper we shall
 exclude all the spaces $SO(2l+1)/U(l), l\geq 3,$ from the above mentioned list (Section \ref{SO}), and quite opposite
 to this, we will prove that all the spaces
$Sp(l)/U(1)\cdot Sp(l-1)=\mathbb{C}P^{2l-1}$ for $l \geq 2$, supplied with invariant Riemannian metrics with the
pinching constants in the open interval $(1/2^4,1/2^2)$, are $\delta$-homogeneous but not normal homogeneous
(Section \ref{SP}).

\begin{remark}
The last result implies the existence of the following unusual geometric situation: for every $l\geq 2$ there are an
irreducible orthogonal representation $r: Sp(l)\rightarrow SO(l(2l+1))$ (actually, the adjoint representation $\Ad$
of $Sp(l)$) in Euclidean space $\mathbb{E}^{l(2l+1)}$ and a convex body $D$ bounded by an ellipsoid (not a ball!)
in $\mathbb{E}^{2(2l-1)}\subset \mathbb{E}^{l(2l+1)}$ such that $D$ is (the image under) the orthogonal projection
(onto $\mathbb{E}^{2(2l-1)}$) of a $r(Sp(l))$-invariant centrally symmetric convex body $B$ in $\mathbb{E}^{l(2l+1)}$.
As a corollary of this, such $B$ cannot be bounded by an ellipsoid in $\mathbb{E}^{l(2l+1)}$
(cf. Remark 2 in \cite{BerNik} for the case $l=2$).
\end{remark}

One can find more precise statements of main results of this paper in the next section.
It is worth to note that the above mentioned method of g.o.-vectors (in particular, $\delta$-vectors) and
characteristic polynomials, which the authors actively apply in the last section of \cite{BerNik} and Sections \ref{SO}
and \ref{SP} in this paper, requires hard calculations, especially calculations of characteristic polynomials of seventh
degree in Section \ref{SO}.

The authors don't know the answers to the following questions.

\begin{quest}
Does there exist a compact simply connected indecomposable $\delta$-homogeneous
but not normal homogeneous Riemannian manifold with zero Euler characteristic?
\end{quest}

\begin{remark}
There are many decomposable Riemannian manifolds of this kind. One can take for example
the direct metric product of $SO(5)/U(2)$, supplied with an invariant
Riemannian metric with the pinching constant in $(1/2^4,1/2^2)$, and the round 3-dimensional sphere.
\end{remark}

\begin{quest}
Is it true that every compact simply connected indecomposable
$\delta$-homogeneous Riemannian manifold is either normal homogeneous, or weakly symmetric?
\end{quest}

\begin{remark}
In the decomposable case the answer to this question is negative.
\end{remark}

D.V.~Alekseevsky and J.A.~Wolf suggested to the authors the following question.

\begin{quest}
Describe all (simply connected) Riemannian manifolds $(M,\mu)$ such that for every point $x\in M$ and every
vector $v\in M_{x}$ there is a Killing vector field $X$ on $(M,\mu)$, attaining the minimal value of its length at $x$,
such that $X(x)=v$.
\end{quest}

\begin{remark}
Any such manifold $(M,\mu)$ is geodesic orbit; it has zero Euler characteristic in the compact case.
The class of manifolds of this kind is closed under the direct metric product operation; it contains
Clifford-Wolf homogeneous Riemannian manifolds and simply connected geodesic orbit (in particular, symmetric)
spaces of nonpositive sectional curvature.
\end{remark}
\medskip
{\bf Acknowledgements.}
The first author is very obliged to Department of Mathematics of University of
Notre Dame, Indiana, USA, for hospitality and a visiting professor position while a
part of this paper has been prepared. The project was supported in part by
the State Maintenance Program for the Leading Scientific Schools
of the Russian Federation (grant NSH-5682.2008.1). The first author was partially supported by RFBR (grant 08-01-00067-a).

\section{Preliminaries}

Here we collect the most important results from \cite{BerNik}, some general
information which we shall need in Sections \ref{SO}, \ref{SP}, and formulate precisely
the main results of this paper.

\begin{prop}[\cite{Al}]
\label{tor}
Let $(M=G/H,\mu)$ be any homogeneous Riemannian manifold and $T$ be any torus in
$H$, $C(T)$ is its centralizer in $G$. Then the orbit $M_T=C(T)(eH)$ is a totally geodesic submanifold of $(M,\mu)$.
\end{prop}

\begin{theorem}[\cite{BerNik}]
\label{tot}
Every closed totally geodesic submanifold of a $\delta$-homogeneous (geodesic orbit)
Riemannian manifold is $\delta$-homogeneous (respectively, geodesic orbit) itself.
\end{theorem}

Now we shall describe a situation, common for both Sections \ref{SO} and \ref{SP}.

Let $G$ be a compact connected Lie group, $H\subset K\subset G$ its closed subgroups.
Fix some $\Ad(G)$-invariant inner product $\langle \cdot, \cdot \rangle$ on the Lie algebra $\mathfrak{g}$
of the group $G$ (recall that there is an unique, up to multiplying by a constant,
such inner product for the case of simple Lie group $G$).
Consider the following $\langle \cdot, \cdot \rangle$-orthogonal decomposition
$$
\mathfrak{g}=\mathfrak{h}\oplus\mathfrak{p}=\mathfrak{h}\oplus\mathfrak{p_1}\oplus\mathfrak{p_2},
$$
where
$$
\mathfrak{k}=\mathfrak{h}\oplus\mathfrak{p_2}
$$
is the Lie algebra of the group $K$.
Obviously, $[\mathfrak{p_2},\mathfrak{p_1}]\subset \mathfrak{p_1}$.
Let $\mu=\mu_{x_1,x_2}$ be a $G$-invariant Riemannian metric on $G/H$, generated by the inner product of the form
\begin{equation}
\label{par} (\cdot,\cdot)=x_1 \langle \cdot ,\cdot
\rangle|_{\mathfrak{p}_1}+ x_2 \langle \cdot ,\cdot
\rangle|_{\mathfrak{p}_2}
\end{equation}
on $\mathfrak{p}$ for positive real numbers $x_1$ and $x_2$.

For any vector $V\in \mathfrak{g}$ we denote by $V_{\mathfrak{h}}$ and $V_{\mathfrak{p}}$  its
$\langle \cdot, \cdot \rangle$-orthogonal projection to $\mathfrak{h}$ and
$\mathfrak{p}$ respectively.

Recall, that the vector $W\in \frak{g}$ is a $\delta$-vector on $(G/H,\mu)$ if
and only if
\begin{equation}
\label{dvect}
(W|_{\mathfrak{p}},W|_{\mathfrak{p}}) \geq (\Ad(a)(W)|_{\mathfrak{p}},\Ad(a)(W)|_{\mathfrak{p}}),
\end{equation}
for every $a\in G$ (see Section 6 in \cite{BerNik}).

\begin{prop}[\cite{BerNik}]
\label{delt}
A homogeneous Riemannian manifold $(G/H,\mu)$ with connected Lie group $G$ is $G$-$\delta$-homogeneous if and only if
for every vector $v\in \mathfrak{p}$ there exists a vector $u\in \mathfrak{h}$ such that the vector $v+u$ is a
$\delta$-vector.
\end{prop}

\begin{prop}[\cite{Tam, BerNik}]
\label{tam}
Let $W=X+Y+Z$ be a geodesic vector on $(G/H,\mu_{x_1,x_2}),$ where $x_1\neq x_2$, $X\in \mathfrak{p_1},
Y\in \mathfrak{p_2}, Z\in \mathfrak{h}$. Then
\begin{equation}
[Z,Y]=0, \quad [X,Y]=x_1/(x_2-x_1)[X,Z].
\end{equation}
\end{prop}

In Section \ref{SO}, we consider the case $(G/H,\mu=\mu_{x_1,x_2}),$ where
$G=SO(2l+1),\quad H=U(l), \quad K=SO(2l), \quad l\geq 3,$ with the embeddings
$U(l)\subset SO(2l)\subset SO(2l+1)$, described below, and $\mu=\mu_{x_1,x_2}$
defined by the inner product (\ref{par}).

For $A,B \in so(2l+1)$ we define $\langle A,B\rangle
=-1/2\trace(A\cdot B)$. This is an $\Ad(SO(2l+1))$-invariant inner
product on $so(2l+1)$.
A matrix $A+\sqrt{-1}B\in u(l)$ we embed into $so(2l)$ via
$A+\sqrt{-1}B \mapsto
\left(
\begin{array}{rr}
A&B\\
-B&A\\
\end{array}
\right)$ in order to get the irreducible symmetric pair $(so(2l),u(l))$ (see
e.g. \cite{Hel}). Also we use the standard embedding $so(2l)$ into
$so(2l+1)$: $A \mapsto \diag(A,0)$.
The inclusions $u(l)\subset so(2l)\subset so(2l+1)$,
constructed above, induce the corresponding inclusions of connected matrix Lie groups
$\tau_l:U(l)\mapsto SO(2l)$ and $\tau'_l:U(l) \mapsto SO(2l+1)$.

The modules $\mathfrak{p}_1$ and $\mathfrak{p}_2$, described above for a general situation,
are $\Ad(\tau'_l(U(l)))$-invariant and $\Ad(\tau'_l(U(l)))$-irreducible in this particular case.

In Section \ref{SP} we find all $\delta$-homogeneous metrics on the spaces
$G/H=Sp(l)/U(1)\cdot Sp(l-1)$, where $H=U(1)\cdot Sp(l-1)\subset
K=Sp(1)\cdot Sp(l-1) \subset Sp(l)$ with embedding described below, and the pairs
$(Sp(l),Sp(1)\cdot Sp(l-1))$, $(Sp(1),U(1))$ are irreducible symmetric.

Let $\mathbb{H}$ be the field of quaternions.
Denote by ${\bf i}$, ${\bf j}$, ${\bf k}$ the quaternionic units in $\mathbb{H}$
(${\bf i}{\bf j}=-{\bf j}{\bf i}={\bf k}$,
${\bf j}{\bf k}=-{\bf k}{\bf j}={\bf i}$,
${\bf k}{\bf i}=-{\bf i}{\bf k}={\bf j}$,
${\bf i}{\bf i}={\bf j}{\bf j}={\bf k}{\bf k}=-1$).
For $X=x_1+{\bf i}x_2+{\bf j}x_3+{\bf k}x_4$, $x_i \in \mathbb{R}$,
define $\overline{X}=x_1-{\bf i}x_2-{\bf j}x_3-{\bf k}x_4$ and
$\|X\|=\sqrt{X\overline{X}}$.
Consider a (left-side) vector space
$\mathbb{H}^l$ over $\mathbb{H}$. For $X=(X_1,X_2,\dots,X_l)\in \mathbb{H}^l$ and
$Y=(Y_1,Y_2,\dots,Y_l)\in \mathbb{H}^l$ we define
$(X,Y)_1=\sum\limits_{s=1}^l X_s \overline{Y}_s$. Then the group $Sp(l)$
is is a group of $\mathbb{R}$-linear operators $A:\mathbb{H}^l \rightarrow \mathbb{H}^l$
with the property $(A(X),A(Y))_1=(X,Y)_1$ for any $X,Y \in \mathbb{H}^l$.
If we choose some $(\cdot,\cdot)_1$-orthonormal quaternionic basis in $\mathbb{H}^l$,
then we can identify $Sp(l)$ with a group of matrices $A=(a_{ij})$,
$a_{ij}\in \mathbb{H}$ with the property $A^{-1}=A^{\ast}$, where
$a_{ij}^{\ast}=\overline{a}_{ji}$ for $1\leq i,j \leq l$.
In this case $sp(l)$ consists of $(l\times l)$-quaternionic matrices $A$ with the property
$A^{\ast}=-A$. Later on we shall use this identifications.

For $A,B \in sp(l)$ we define

\begin{equation}\label{product}
\langle A,B\rangle=\frac{1}{2} \trace( \operatorname{Re} (AB^{\ast})).
\end{equation}

It is easy to see that $\langle \cdot,\cdot\rangle$ is a $\Ad(Sp(l))$-invariant
inner product on the Lie algebra $\mathfrak{g}=sp(l)$.
In the sequel we shall suppose (without loss of generality) that the embedding of
$sp(1)\oplus sp(l-1)$ in $sp(l)$ is defined by
$(A,B) \mapsto \diag(A,B)$, where $A\in sp(1)$ and $B\in sp(l-1)$.
It is clear that the modulus $\mathfrak{p}_1$ and $\mathfrak{p}_2$ are
$\Ad(Sp(l))$-invariant and $\Ad(Sp(l))$-irreducible. We know that every invariant Riemannian metric
$\mu={\mu}_{x_1,x_2}$ on $Sp(l)/U(1)\cdot Sp(l-1)$, corresponding to the inner product
(\ref{par}), is a g.o.-metric \cite{Zi96}.

One of the main results of the paper \cite{BerNik} is the following

\begin{theorem}[\cite{BerNik}]
\label{oot}
No one of compact simply connected (connected) indecomposable homogeneous Riemannian
manifolds with positive Euler characteristic, excepting $(SO(5)/U(2)=Sp(2)/U(1)\cdot Sp(1), \mu={\mu}_{x_1,x_2})$,
and possibly  $(SO(2l+1)/U(l), \mu={\mu}_{x_1,x_2})$ or
$(Sp(l)/U(1)\cdot Sp(l), \mu={\mu}_{x_1,x_2})$, where $l\geq 3$ and $x_1< x_2< 2x_1$
in all the cases above, cannot be $\delta$-homogeneous but not normal homogeneous
Riemannian manifold.
\end{theorem}

The main result of Section \ref{SO} is the following

\begin{theorem}\label{osn1}
The Riemannian manifold $(SO(2l+1)/U(l),\mu={\mu}_{x_1,x_2}),$ where $l\geq 3$, is not
$\delta$-homogeneous if $x_1<x_2<2x_1$.
\end{theorem}

\begin{prop}[\cite{On}]\label{vvspom0}
The full connected isometry group of $(Sp(l)/U(1)\cdot Sp(l-1)
,\mu)$ is $Sp(l)/\{\pm I\}$, excepting the case $x_2=2x_1$, where
the full connected isometry group is a factor-group of $SU(2l)$ by
its center, and the metric $\mu$ is $SU(2l)$-normal (in the last
case $(Sp(l)/U(1)\cdot Sp(l-1),\mu)$ is isometric to the complex
projective space $\mathbb{C}P^{2l-1}=SU(2l)/U(1)\cdot
S(U(2l-1))$).
\end{prop}

The main result of Section \ref{SP} is the following

\begin{theorem}\label{main1}
The Riemannian manifold $(Sp(l)/U(1)\cdot
Sp(l-1),\mu={\mu}_{x_1,x_2})$ is $\delta$-homogeneous if and only
if $x_1\leq x_2 \leq 2x_1$. For $x_2=x_1$ it is $Sp(l)$-normal
homogeneous; for $x_2=2x_1$ it is $SU(2l)$-normal homogeneous; for
$x_2\in (x_1,2x_1)$ it is not normal homogeneous with respect to
any its isometry group.
\end{theorem}

\section{The spaces $SO(2l+1)/U(l)$, $l\geq 3$}\label{SO}

At first we will show that the Riemannian manifold $(SO(7)/U(3),\mu=\mu_{x_1,x_2})$, $x_1<x_2<2x_1$,
is not $\delta$-homogeneous.

Using the above notation, we have in this particular case
$$
A=\left(
\begin{array}{rrr}
0&a&b\\
-a&0&c\\
-b&-c&0\\
\end{array}
\right), \quad
B=\left(
\begin{array}{rrr}
d&e&f\\
e&g&h\\
f&h&k\\
\end{array}
\right),
$$

$$
u(3)=\left\{ \left(
\begin{array}{rrrrrrr}
0&a&b&d&e&f&0\\
-a&0&c&e&g&h&0\\
-b&-c&0&f&h&k&0\\
-d&-e&-f&0&a&b&0\\
-e&-g&-h&-a&0&c&0\\
-f&-h&-k&-b&-c&0&0\\
0&0&0&0&0&0&0\\
\end{array}
\right) \,;\quad a,b,c,d,e,f,g,h,k \in \mathbb{R}\, \right\},
$$
$$
\mathfrak{p}_1=\left\{ X= \left(
\begin{array}{rrrrrrr}
0&0&0&0&0&0&s_1\\
0&0&0&0&0&0&s_2\\
0&0&0&0&0&0&s_3\\
0&0&0&0&0&0&s_4\\
0&0&0&0&0&0&s_5\\
0&0&0&0&0&0&s_6\\
-s_1&-s_2&-s_3&-s_4&-s_5&-s_6&0\\
\end{array}
\right) \,;\quad s_i \in \mathbb{R}\, \right\},
$$
$$
\mathfrak{p}_2=\left\{ Y=\left(
\begin{array}{rrrrrrr}
0&l&m&0&p&q&0\\
-l&0&n&-p&0&r&0\\
-m&-n&0&-q&-r&0&0\\
0&p&q&0&-l&-m&0\\
-p&0&r&l&0&-n&0\\
-q&-r&0&m&n&0&0\\
0&0&0&0&0&0&0\\
\end{array}
\right) \,;\quad l,m,n,p,q,r \in \mathbb{R}\, \right\}.
$$
Note that for vectors $X$ from $\mathfrak{p}_1$ as above we have $\langle
X,X\rangle =s_1^2+s_2^2+s_3^2+s_4^2+s_5^2+s_6^2$, and for vectors $Y\in
\mathfrak{p}_2$ we have $\langle Y,Y\rangle =2l^2+2m^2+2n^2+2p^2+2q^2+2r^2$.

Let $E_{i,j}$ be a $(7\times 7)$-matrix, whose $(i,j)$-th entry is
equal to $1$, and all other entries are zero. For any $1\leq i <j\leq 7$,
we put $F_{i,j}=E_{i,j}-E_{j,i}$.

\begin{prop}\label{osn}
The Riemannian manifold $(SO(7)/U(3),\mu={\mu}_{x_1,x_2})$ is not
$\delta$-homogeneous if $x_1<x_2<2x_1$.
\end{prop}

For $W\in so(7)$, we denote by $O(W)$ the orbit of $W$ under the action of
$\Ad(SO(7))$, i.e.
$$
O(W)=\{V\in so(7)\,|\,~ \exists Q\in SO(7), V=QWQ^{-1}\}.
$$

\begin{lemma}\label{vspom1}
Let $W=X+Y+Z$, where $X=s_1F_{1,7}\in \mathfrak{p}_1$ ($s_1\neq 0$),
$Y=q(F_{1,6}-F_{3,4})+r(F_{2,6}-F_{3,5})\in \mathfrak{p}_2$ ($q\neq 0$, $r\neq 0$),
$Z\in
\mathfrak{h}=u(3)$ (see above), be a geodesic vector on
$(SO(7)/U(3),\mu)$ for $x_1<x_2<2x_1$. Then

\begin{multline*}
W=\left(
\begin{array}{ccccccr}
0 & 0 & 0 & 0 & 0 & \frac{x_2}{x_1}q & s_1 \\
0 & 0 & 0 & 0 & 0 & \frac{x_2}{x_1}r & 0 \\
0 & 0 & 0 & \frac{x_2-2x_1}{x_1}q & \frac{x_2-2x_1}{x_1}r & 0 & 0 \\
0 & 0 & \frac{2x_1-x_2}{x_1}q & 0 & 0 & 0 & 0 \\
0 & 0 & \frac{2x_1-x_2}{x_1}r & 0 & 0 & 0 & 0 \\
-\frac{x_2}{x_1}q & -\frac{x_2}{x_1}r & 0 & 0 & 0 & 0 & 0 \\
-s_1 & 0 & 0 & 0 & 0 & 0 & 0
\end{array}
 \right)=
\\
=s_1F_{1,7}+\frac{x_2}{x_1}qF_{1,6}+\frac{x_2}{x_1}rF_{2,6}+\frac{x_2-2x_1}{x_1}qF_{3,4}+\frac{x_2-2x_1}{x_1}rF_{3,5}.
\end{multline*}
\end{lemma}

\begin{proof}
Since $W$ is geodesic vector, then from Proposition \ref{tam} we
get $[Z,Y]=0$, $[X,Y]=x_1/(x_2-x_1)[X,Z]$.
Direct calculations show that
$$[Z,Y]=(qh-fr)(F_{1,2}-F_{4,5})+(qd+qk+er)(F_{1,3}-F_{4,6})+$$
$$(rk+rg+eq)(F_{2,3}-F_{5,6})+(cq-br)(F_{1,5}-F_{2,4})+ar(F_{1,6}-F_{3,4})+aq(F_{3,5}-F_{2,6}),$$
$$[X,Y]=s_1qF_{6,7},[X,Z]=s_1(aF_{2,7}+bF_{3,7}+dF_{4,7}+eF_{5,7}+fF_{6,7}).$$
The vectors $F_{1,2}-F_{4,5}$, $F_{1,3}-F_{4,6}$, $F_{2,3}-F_{5,6}$, $F_{1,5}-F_{2,4}$, $F_{1,6}-F_{3,4}$,
$F_{3,5}-F_{2,6}$ are linearly independent in $\mathfrak{p}_2$, and the vectors $F_{i,7}$, $2\leq i \leq 6$,
are linearly independent in $\mathfrak{p}_1$.
Therefore, $a=b=d=e=c=k=g=0$,
$f=\frac{x_2-x_1}{x_1}q$ and $h=\frac{x_2-x_1}{x_1}r$.
The lemma is proved.
\end{proof}

\begin{remark}
\label{r1}
Note, that in the paper \cite{Du}, the structure of all geodesic vectors on \linebreak
$(SO(7)/U(3),\mu)$
is studied.
\end{remark}

\begin{lemma}\label{vspom2}
If the Riemannian manifold $(SO(7)/U(3),\mu)$, $x_1<x_2<2x_1$, is
$SO(7)$-$\delta$-homogeneous then for every $s_1\neq 0$, $q\neq 0$, $r\neq 0$ the vector
$$
W=s_1F_{1,7}+\frac{x_2}{x_1}qF_{1,6}+\frac{x_2}{x_1}rF_{2,6}+\frac{x_2-2x_1}{x_1}qF_{3,4}+\frac{x_2-2x_1}{x_1}rF_{3,5}
$$
is $\delta$-vector on $(SO(7)/U(3),\mu)$.
\end{lemma}

\begin{proof}
If $(SO(7)/U(3),\mu)$ is $SO(7)$-$\delta$-homogeneous, then for
every vector of the form $V=X+Y$, where  $X=s_1F_{1,7}\in \mathfrak{p}_1$ ($s_1\neq 0$),
$Y=q(F_{1,6}-F_{3,4})+r(F_{2,6}-F_{3,5})\in \mathfrak{p}_2$ ($q\neq 0$, $r\neq 0$),
there is $Z\in \mathfrak{h}$ such that the vector
$\widetilde{W}=X+Y+Z$ is $\delta$-vector (see Proposition \ref{delt}).
In particular, such $\widetilde{W}$ should be a geodesic
vector. According to Lemma \ref{vspom1}, we get that
$$
\widetilde{W}=W=s_1F_{1,7}+\frac{x_2}{x_1}qF_{1,6}+\frac{x_2}{x_1}rF_{2,6}+
\frac{x_2-2x_1}{x_1}qF_{3,4}+\frac{x_2-2x_1}{x_1}rF_{3,5}.
$$
Therefore, this $W$ is a $\delta$-vector.
\end{proof}

\begin{lemma}\label{vspom3}
Let $A,B\in so(7)$. Then $A,B$ are in the same orbit of $\Ad(SO(7))$ if and only if their
characteristic polynomials coincide.
\end{lemma}

\begin{proof}
It is obvious that if $A$ and $B$ are in the same orbit of $\Ad(SO(7))$,
then their characteristic polynomials coincide.

Suppose, that characteristic polynomials of $A$ and $B$ are coincide. The standard
Weyl chamber of the Lie algebra $so(7)$ is the following (see \cite{Bes}):
$$
K=\left\{ \left.\diag\left(\left(
\begin{array}{rr}
0&-z_1\\
z_1&0\\
\end{array}
\right),
\left(
\begin{array}{rr}
0&-z_2\\
z_2&0\\
\end{array}
\right),
\left(
\begin{array}{rr}
0&-z_3\\
z_3&0\\
\end{array}
\right)
,0
\right)\right|~z_1\geq z_2\geq z_3\geq 0 \right\}.
$$

If $A$ and $B$ are conjugate to distinct elements of the Weyl chamber, then,
as it is easy to see, their characteristic polynomials are distinct.
Hence, $A$ and $B$ are conjugate to one and the same element of the Weyl chamber.
This implies that $A$ and $B$ are in one and the same orbit of $\Ad(SO(7))$.
The lemma is proved.
\end{proof}

In what follows we need the value $\lambda= \frac{x_2}{x_1}$.
Now we consider the following two geodesic vectors $W$ and $\widetilde{W}$ (see Lemma \ref{vspom1}):
\begin{equation}\label{vect1}
W=s_1F_{1,7}+\frac{x_2}{x_1}qF_{1,6}+\frac{x_2}{x_1}rF_{2,6}+\frac{x_2-2x_1}{x_1}qF_{3,4}+\frac{x_2-2x_1}{x_1}rF_{3,5},
\end{equation}
where
$$
s_1=\sqrt{\lambda^2((2-\lambda)^2+\lambda^3-1)},
q=\sqrt{\frac{(\lambda^3-1)(1-(2-\lambda)^2)}{(2-\lambda)^2+\lambda^3-1}},
r=\sqrt{\frac{\lambda^3(2-\lambda)^2}{(2-\lambda)^2+(\lambda^3-1)}};
$$
\begin{equation}\label{vect2}
\widetilde{W}=\widetilde{s}_1F_{1,7}+\frac{x_2}{x_1}\widetilde{q}F_{1,6}+\frac{x_2}{x_1}\widetilde{r}F_{2,6}
+\frac{x_2-2x_1}{x_1}\widetilde{q}F_{3,4}+\frac{x_2-2x_1}{x_1}\widetilde{r}F_{3,5},
\end{equation}
where
$$
\widetilde{s}_1=\sqrt{(2-\lambda)^2+\lambda^4(\lambda-1)},
\widetilde{q}=\sqrt{\frac{\lambda^2(\lambda-1)(\lambda^4-(2-\lambda)^2)}{(2-\lambda)^2+\lambda^4(\lambda-1)}},
\widetilde{r}=\sqrt{\frac{\lambda^3(2-\lambda)^2}{(2-\lambda)^2+\lambda^4(\lambda-1)}}.
$$

\begin{lemma}\label{vspom4}
The vector $W$ (see (\ref{vect1})) is not a $\delta-$vector on
$(SO(7)/U(3),\mu)$ for $x_1<x_2<2x_1$.
\end{lemma}

\begin{proof}
Direct calculations show that the characteristic polynomials $P(z)$ and $\widetilde{P}(z)$
of the matrices $W$ and $\widetilde{W}$ (see (\ref{vect2})) are the following:
$$
P(z)=z^7+
(a+b(\lambda^2+(2-\lambda)^2))z^5+
(ab(2-\lambda)^2+ac\lambda^2+b^2\lambda^2(2-\lambda)^2)z^3
+abc\lambda^2(2-\lambda)^2z,
$$
$$
\widetilde{P}(z)=z^7+
(\widetilde{a}+\widetilde{b}(\lambda^2+(2-\lambda)^2))z^5+
(\widetilde{a}\widetilde{b}(2-\lambda)^2+\widetilde{a}\widetilde{c}\lambda^2+\widetilde{b}^2\lambda^2(2-\lambda)^2)z^3
+\widetilde{a}\widetilde{b}\widetilde{c}\lambda^2(2-\lambda)^2z,
$$
where
$$
\lambda= \frac{x_2}{x_1},\, a=s_1^2,\, b=q^2+r^2,\, c=r^2,\,
\widetilde{a}=\widetilde{s}_1^2, \, \widetilde{b}=\widetilde{q}^2+\widetilde{r}^2,\, \widetilde{c}=\widetilde{r}^2.
$$

Now, we shall show that
$P(z)=\widetilde{P}(z)$
and
$(W|_\mathfrak{p},W|_\mathfrak{p})<(\widetilde{W}|_\mathfrak{p},\widetilde{W}|_\mathfrak{p})$.
Since $x_1<x_2<2x_1$, then $1<\lambda<2$.
It is easy to check that
$$
b=1, \quad a=\lambda^2((2-\lambda)^2+\lambda^3-1), \quad
c=\frac{\lambda^3(2-\lambda)^2}{(2-\lambda)^2+(\lambda^3-1)},
$$
$$
\widetilde{b}=\lambda^2,  \quad\widetilde{a}=(2-\lambda)^2+\lambda^4(\lambda-1), \quad
\widetilde{c}=\frac{\lambda^3(2-\lambda)^2}{(2-\lambda)^2+\lambda^4(\lambda-1)}.
$$
The equality $P(z)=\widetilde{P}(z)$ is equivalent to the
following system of equations:

\begin{equation}\label{sys}
\left\{\begin{array}{l}
    a+b(\lambda^2+(2-\lambda)^2)=\widetilde{a}+\widetilde{b}(\lambda^2+(2-\lambda)^2), \\
    ab(2-\lambda)^2+ac\lambda^2+b^2\lambda^2(2-\lambda)^2=
    \widetilde{a}\widetilde{b}(2-\lambda)^2+\widetilde{a}\widetilde{c}\lambda^2+\widetilde{b}^2\lambda^2(2-\lambda)^2, \\
    abc\lambda^2(2-\lambda)^2=\widetilde{a}\widetilde{b}\widetilde{c}\lambda^2(2-\lambda)^2. \\
  \end{array}\right.
\end{equation}

It is easy to verify, that system (\ref{sys}) is fulfilled for the considered
$a,b,c,\widetilde{a},\widetilde{b},\widetilde{c}$.
Therefore, $P(z)=\widetilde{P}(z)$.

Since $(W|_\mathfrak{p},W|_\mathfrak{p})=x_1(a+2\lambda b)$ and
$(\widetilde{W}|_\mathfrak{p},\widetilde{W}|_\mathfrak{p})=x_1(\widetilde{a}+2\lambda\widetilde{b})$,
then the inequality
$(W|_\mathfrak{p},W|_\mathfrak{p})<(\widetilde{W}|_\mathfrak{p},\widetilde{W}|_\mathfrak{p})$
is equivalent to the following one:
$a+2\lambda b<\widetilde{a}+2\lambda\widetilde{b}$. It is easy to see, that
$$
\widetilde{a}+2\lambda\widetilde{b}-a-2\lambda b=
(2-\lambda)^2+\lambda^4(\lambda-1)+2\lambda^3-\lambda^2((2-\lambda)^2+\lambda^3-1)-2\lambda=$$
$$
2(2-\lambda)(\lambda^2-1)(\lambda-1)>0.
$$
Therefore, $(W|_\mathfrak{p},W|_\mathfrak{p})<(\widetilde{W}|_\mathfrak{p},\widetilde{W}|_\mathfrak{p})$.

Since $P(z)=\widetilde{P}(z)$, then by Lemma \ref{vspom3} we get $\widetilde{W} \in O(W)$.
On the other hand, $(W|_\mathfrak{p},W|_\mathfrak{p})<(\widetilde{W}|_\mathfrak{p},\widetilde{W}|_\mathfrak{p})$.
Consequently, the vector $W$ is not a $\delta$-vector, because otherwise the inequality
$$
(W|_{\mathfrak{p}},W|_{\mathfrak{p}}) \geq (\widetilde{W}|_{\mathfrak{p}},\widetilde{W}|_{\mathfrak{p}})
$$
must hold (see the formula (\ref{dvect}) for $\delta$-vectors above).
The lemma is proved.
\end{proof}

Now, it suffices to note that {\bf the proof of Proposition \ref{osn}} follows from
Lemma \ref{vspom2} and Lemma~\ref{vspom4}.

For $1\leq m < l$, we define
the embedding $\sigma_{m,l}: SO(2m+1)\times SO(2k) \mapsto SO(2l+1)$, where $k=l-m$.
This embedding is completely determined by the embedding $d\sigma_{m,l}: so(2m+1)\oplus so(2k)\mapsto so(2l+1)$
for the corresponding Lie algebras. Note that
$so(2m+1)$ consists of matrices of the following type
$$
Q_1=\left(
\begin{array}{ccc}
V&U&E\\
-U^t&W&F\\
-E^t&-F^t&0\\
\end{array}
\right),
$$
where $V$ and $W$ are skew-symmetric $(m\times m)$-matrices, $U$ is an arbitrary $(m\times m)$-matrix,
$E$ and $F$ are arbitrary $(m\times 1)$-matrices.
The Lie algebra $so(k)$ consists of matrices of the following form
$$
Q_2=\left(
\begin{array}{ccc}
A&B\\
-B^t&C\\
\end{array}
\right),
$$
where $A$ and $C$ are skew-symmetric $(k\times k)$-matrices and $B$ is an arbitrary $(k\times k)$-matrix,
Now we define $d\sigma_{m,l}$ as follows

$$
d\sigma_{m,l}((Q_1,Q_2))=\left(
\begin{array}{ccccc}
V&O&U&O&E\\
O&A&O&B&O\\
-U^t&O&W&O&F\\
O&-B^t&O&C&O\\
-E^t&O&-F^t&O&0\\
\end{array}
\right),
$$
where $O$'s denote zero matrices.

Note, that for the considered embeddings we have
$$
\sigma_{m,l}\Bigl(\tau'_m(U(m))\times \tau_k(U(k))\Bigr)\subset \tau'_l(U(l)),
$$
$$
\sigma_{m,l}\Bigl(\tau'_m(U(m))\times \Id \Bigr)=\sigma_{m,l}\Bigl(SO(2m+1)\times \Id \Bigr)\cap \tau'_l(U(l)).
$$

Now we suppose that $l\geq 3$ and $1<m <l$. Let us consider $G=SO(2l+1)$, $H=\tau'_l(U(l))$,
$\widetilde{G}=\sigma_{m,l}\Bigl(SO(2m+1)\times \Id \Bigr)$, and
$\widetilde{H}=\sigma_{m,l}\Bigl(\tau'_m(U(m))\times \Id \Bigr)$.

It is clear that
\begin{equation}
\label{emb}
\widetilde{G} \subset G,\quad \widetilde{H}=\widetilde{G}\cap H;\quad \widetilde{G}\cap SO(2l)=\sigma_{m,l}(SO(2m)),
\end{equation}

\begin{equation}
\label{emb1}
d\sigma_{m,l}(so(2m)^\perp)=d\sigma_{m,l}(so(2m+1))\cap (so(2l))^\perp,
\end{equation}
and
\begin{equation}
\label{emb2}
d\sigma_{m,l}(d\tau_m(u(m))^\perp)=d\sigma_{m,l}(so(2m+1))\cap (d\tau_l(u(l)))^\perp.
\end{equation}

\begin{lemma}\label{lem0}
The orbit of the group $\widetilde{G}$ through the point
$\bar{e}=eH$ in $(G/H,\mu={\mu}_{x_1,x_2})$, that is $\widetilde{G}/\widetilde{H},$
supplied with the induced Riemannian metric $\eta$, is a totally geodesic submanifold of $(G/H,\mu={\mu}_{x_1,x_2})$.
Moreover, the map $(SO(2m+1)/U(m),{\mu}_{x_1,x_2})\rightarrow (\widetilde{G}/\widetilde{H},\eta)$ is an isometry.
\end{lemma}

\begin{proof}
For this goal let us consider $T$, a maximal ($k$-dimensional) torus in \linebreak
$\sigma_{m,l}\Bigl(\Id\times \tau_k(U(k))\Bigr)$. Note, that $T$ is also a maximal torus in
$\sigma_{m,l}\Bigl(\Id\times SO(2k)\Bigr)$ and $T\subset H$. Let $C$ be the centralizer of $T$ in $SO(2l+1)$.
It is easy to see that $C=T\cdot \widetilde{G}$.
By Proposition \ref{tor}, the orbit of $C$ through the point $eH\in G/H$ is a totally geodesic submanifold of
$(G/H,\mu={\mu}_{x_1,x_2}))$ with the induced Riemannian
metric $\eta$. But $T\subset H$ and, consequently, this orbit coincides with the space
$\widetilde{G}/\widetilde{H}$. The inclusions (\ref{emb}), (\ref{emb1}), and (\ref{emb2})
imply that the map $(SO(2m+1)/U(m),{\mu}_{x_1,x_2})\rightarrow (\widetilde{G}/\widetilde{H},\eta)$ is an isometry.
\end{proof}

\begin{proof}[Proof of Theorem \ref{osn1}]
The case $l=3$ has been concidered in Proposition \ref{osn}. Let us suppose that $l\geq 4$
and $(SO(2l+1)/U(l),\mu={\mu}_{x_1,x_2})$, where $x_1<x_2<2x_1$, is $\delta$-homogeneous.
Then by Lemma \ref{lem0}, $(SO(7)/U(3),\mu={\mu}_{x_1,x_2})$ is a totally geodesic submanifold of
the $\delta$-homogeneous manifold $(SO(2l+1)/U(l),\mu={\mu}_{x_1,x_2})$,
and by Theorem \ref{tot}, it must be $\delta$-homogeneous itself. We get a contradiction
with Proposition \ref{osn}.
\end{proof}

\section{The spaces $Sp(l)/U(1)\cdot  Sp(l-1)$, $l\geq 2$}\label{SP}

Let us consider a Lie subalgebra $\widetilde{\mathfrak{g}}$ in $\mathfrak{g}=sp(l)$ of the
form
$$
\widetilde{\mathfrak{g}}=\{ \diag(A,0)\in sp(l)\,|\, A\in sp(2),\,0 \in sp(l-2) \}.
$$
Let $\widetilde{G}=Sp(2)$ be a connected (closed) subgroup of $G=Sp(l)$
corresponding to $\widetilde{\mathfrak{g}}$, and $\widetilde{H}=\widetilde{G}\cap H$.

\begin{lemma}\label{lem1}
The orbit of the group $\widetilde{G}$ through the point
$\bar{e}=eH$ in $(G/H,\mu={\mu}_{x_1,x_2})$, that is $\widetilde{G}/\widetilde{H}$, is totally geodesic submanifold.
\end{lemma}

\begin{proof}
Let us consider a torus
$T=\diag(1,1,S_1,\dots, S_{l-2})\subset Sp(l)$, where $S_i$ is a circle subgroup.
It is easy to see that $T\subset H$ and
$\widetilde{G} \times T$ is a connected component (of the unit)
of the centralizer of $T$. It follows from Proposition \ref{tor}, that
the orbit of this subgroup trough the point $eH$ is a totally geodesic submanifold
in $(G/H,\mu)$. But this orbit coincides with $\widetilde{G}/\widetilde{H}$.
\end{proof}

It is clear that $\widetilde{H}=U(1)\times Sp(1)$,
where $U(1)\times Sp(1)\subset Sp(1)\times Sp(1) \subset Sp(2)=\widetilde{G}$.
Therefore we have the following $\langle \cdot , \cdot \rangle$-orthogonal
decomposition for corresponded Lie algebras:
\begin{equation}\label{small}
\widetilde{\mathfrak{g}}=sp(2)=\widetilde{\mathfrak{h}}\oplus \mathfrak{p}_2 \oplus
\mathfrak{p}_1^{\prime},
\end{equation}
where $\mathfrak{p}_1^{\prime}=\mathfrak{p}_1\cap \widetilde{\mathfrak{g}}$.

\begin{lemma}\label{lem2}
Let $X\in \mathfrak{p}_1^{\prime}$, $Y\in \mathfrak{p}_2$ be some nontrivial vectors.
Then for any $Z\in \mathfrak{p}$ there is
$a \in H$ such that $\Ad(a)(Z)=cX+dY$ for some $c,d \geq 0$.
\end{lemma}

\begin{proof}
Let $Z=Z_1+Z_2$, where $Z_1\in \mathfrak{p}_1$, $Z_2 \in \mathfrak{p}_2$.
Recall that $U(1)$ acts on $\mathfrak{p}_2$ and $\mathfrak{p}_1$ by rotations.
Hence we can find $a_1\in U(1)$ such that $\Ad(a_1)(Z_2)=dY$ for some nonnegative $d$.
Further, recall
that $\mathbb{H}P^{l-1}=Sp(l)/Sp(1)\times Sp(l-1)$ is two-point homogeneous.
Therefore, there is $a_2\in Sp(1)\times Sp(l-1)$ such that
$\Ad(a_2)(Z_1^{\prime})=c X$ for some $c\geq 0$, where $Z_1^{\prime}=\Ad(a_2)(Z_1)$.
Moreover, such $a_2$ can be chosen from $Sp(l-1)$, since already $Sp(l-1)$
acts transitively on the unit sphere in $\mathbb{H}P^{l-1}$ (see e.g. \cite{Zi96}).
Therefore, one can choose $a=a_2\cdot a_1$.
\end{proof}

We write $E_{ij}$ for the skew-symmetric matrix with $1$ in the $ij$-th entry and
$-1$ in the $ji$-th entry, and zeros elsewhere.
We denote by $F_{ij}$ the symmetric matrix with $1$ in both
the $ij$-th and $ji$-th entries, and zeros elsewhere.
Denote also by $G_i$ the matrix with $\sqrt{2}$
in $ii$-th entry, and zeros elsewhere.

It is easy to check that the matrices of the forms
${\bf i}G_i$, ${\bf j}G_i$, ${\bf k}G_i$, $E_{ij}$, ${\bf i}F_{ij}$,
${\bf j}F_{ij}$, ${\bf k}F_{ij}$, where $1\leq i, j \leq n$ and $i<j$,
form a $\langle \cdot, \cdot \rangle$-orthonormal (see (\ref{product}))
basis in $sp(l)$.

Without loss of generality we may suppose that the Lie subalgebra $u(1)$
($\mathfrak{h}=u(1)\oplus sp(l-1)$)
is spanned on the vector ${\bf i}G_1$.
It is clear that $E_{12}\in \mathfrak{p}_1^{\prime}$ and ${\bf j}G_1 \in \mathfrak{p}_2$.

\begin{lemma}\label{lem3}
Let $W=X+Y+Z$ be a $\delta$-vector on $\widetilde{G}/\widetilde{H}$ with a
metric induced by $\mu$, where $X=cE_{12}$ and $Y=d{\bf j}G_1$ for some non-negative
$c$ and $d$. Then the following relations hold:

1) If $c=0$, then $Z=\beta {\bf i}G_2+\gamma {\bf j}G_2+\delta {\bf k}G_2$,
$\beta,\gamma, \delta \in \mathbb{R}$;

2) If $d=0$, then $Z=\alpha ({\bf i} G_1 +{\bf i} G_2)$,
$\alpha \in \mathbb{R}$;

3) If $c\neq 0$ and $d\neq 0$, then
$Z=-\frac{x_2-x_1}{x_1}d{\bf j} G_2$.
\end{lemma}

\begin{proof}
The vector $W$ is g.o.-vector. According to Proposition \ref{tam}, we have
$[Z,Y]=0$ and $[Z,X]=\frac{x_2-x_1}{x_1}[Y,X]$. Consider an arbitrary
$Z_1=\alpha {\bf i} G_1+\beta {\bf i}G_2+\gamma {\bf j}G_2+\delta {\bf k}G_2 \in
\widetilde{h}$. It is easy to see that
$[Z_1,Y]=2\sqrt{2}\alpha d {\bf k} G_1$,
$[Z_1,X]=\sqrt{2}c((\alpha-\beta) {\bf i} F_{12}-\gamma {\bf j}F_{12}-\delta {\bf k} F_{12})$,
$[Y,X]=\sqrt{2}c d{\bf j}F_{12}$. These formulas imply all statements of Lemma.
\end{proof}

Consider now the vectors $X=cE_{12}$ and $Y=d{\bf j}G_1$ for some positive $c$ and $d$.
It is easy to see that the vector $Z=-\frac{x_2-x_1}{x_1} d{\bf j} G_2$ satisfies
the relations $[Z,Y]=0$, $[Z,X]=\frac{x_2-x_1}{x_1} [Y,X]$.
Indeed, the vector $W=X+Y+Z$ is a $\delta$-vector on the space
$\widetilde{G}/\widetilde{H}$ with a metric induced by $\mu_{x_1,x_2}$, if
$x_1 <x_2 \leq 2x_1$ (see Section 13 in \cite{BerNik}).

Our main technical tool is the following

\begin{prop}\label{osn2}
If for every positive $c$ and $d$ the vector
$$
W=X+Y+Z=cE_{12}+d{\bf j}G_1-\frac{x_2-x_1}{x_1}d{\bf j} G_2
$$
is a $\delta$-vector on $(G/H, \mu)$, then the Riemannian manifolds
$(G/H, \mu)$ is $G$-$\delta$-homogeneous.
\end{prop}

\begin{proof}
We recall that a $\delta$-vector $W\in \frak{g}$ is characterized by the
equation (\ref{dvect}), and $(G/H,\mu)$
is $\delta$-homogeneous if any vector from $\frak{p}$ can be represented
as $W|_{\frak{p}}$ for some $\delta$-vector $W\in \frak{g}$ (see Proposition \ref{delt}).
Evidently, $(\Ad(h)(W)|_{\frak{p}},\Ad(h)(W)|_{\frak{p}})=(W|_{\frak{p}},W|_{\frak{p}})$ for all
$W\in \frak{g}$ and $h\in H$. This fact, together with Lemma \ref{lem2}, implies the Proposition.
\end{proof}

\begin{prop}\label{osn3}
Suppose that
$$
W=X+Y+Z=cE_{12}+d{\bf j}G_1-\frac{x_2-x_1}{x_1}d{\bf j} G_2
$$
is not a $\delta$-vector on $(G/H, \mu)$.
Then there is a vector $\widetilde{W}$ in the $\Ad(G)$-orbit of $W$, which has
one of the following forms:

1) $\widetilde{W}_1=\widetilde{d} {\bf j}G_1  +\sum\limits_{q=2}^l {\alpha}_q {\bf i} G_q$,
where $x_2\widetilde{d}^2 > x_2 d^2 +x_1 c^2$;

2) $\widetilde{W}_2=\widetilde{c} E_{12}  +
\alpha ({\bf i} G_1 +{\bf i} G_2)+\sum\limits_{q=3}^l {\alpha}_q {\bf i} G_q$,
where $x_1\widetilde{c}^2 > x_2 d^2 +x_1 c^2$;

3) $\widetilde{W}_3=
\widetilde{c}E_{12}+\widetilde{d}{\bf j}G_1-\frac{x_2-x_1}{x_1}\widetilde{d}{\bf j} G_2
+\sum\limits_{q=3}^l {\alpha}_q {\bf i} G_q$,
where $x_2\widetilde{d}^2+x_1\widetilde{c}^2 > x_2 d^2 +x_1 c^2$.

In the formulas above $\alpha, {\alpha}_q \in \mathbb{R}$,
$\widetilde{c},\widetilde{d} \geq 0$.
\end{prop}

\begin{proof}
If $W$ is not a $\delta$-vector, then
$$
M:=\max_{a\in G} (\Ad(a)(W)|_{\mathfrak{p}},\Ad(a)(W)|_{\mathfrak{p}}) >
(W|_{\mathfrak{p}},W_{\mathfrak{p}})=x_2d^2+x_1c^2.
$$
Consider some $\widetilde{W}$ from the $\Ad(G)$-orbit of $W$, which gives the maximal value
$M$ in the above formula, then $\widetilde{W}$ is a $\delta$-vector on $(G/H,\mu)$.
Using Lemma \ref{lem2},
we may assume that
$\widetilde{W}_{\mathfrak{p}}=\widetilde{X}+\widetilde{Y}$, where
$\widetilde{X}=\widetilde{c}E_{12}$ and $\widetilde{Y}=\widetilde{d}{\bf j}G_1$
for same nonnegative
$\widetilde{c}$ and $\widetilde{d}$. Consider now $W_{\mathfrak{h}} =Z_1+Z_2+Z_3$,
where $Z_1 \in \widetilde{\mathfrak{h}}$, $Z_2\in sp(l-2)$, $Z_3\in \mathfrak{p}_3$,
$\mathfrak{p}_3$ is a $\langle \cdot, \cdot \rangle$-compliment to $sp(l-2)$ in $sp(l-1)$,
and $sp(l-1)$ ($sp(l-2)$) is defined by the embedding $X\rightarrow \diag(0,X)$
(respectively, $X\rightarrow \diag(0,0,X)$) to $sp(l)$.

It is well-known that if we interpret any element $U\in \frak{g}$ as a right-invariant vector field on $G$,
then $X=d\pi(U)$, where $\pi: G\rightarrow G/H$ is the natural projection, correctly defines a Killing vector
field on  $(G/H,\mu)$. Under this
$U$ is a $\delta$-vector if and only if $X$ attains the maximal value of its length at the initial point
$eH\in G/H$ \cite{BerNik}. Since $\widetilde{G}/\widetilde{H}$ is totally geodesic submanifold of $(G/H,\mu)$
by Lemma \ref{lem1}, the proof of Theorem \ref{tot} (Theorem 11 in \cite{BerNik}) implies that the tangent to
$\widetilde{G}/\widetilde{H}$ component of
such field $X$ is a Killing vector field on $\widetilde{G}/\widetilde{H}$, which also
attain the maximal value of its length at the initial point $e\widetilde{H} \in \widetilde{G}/\widetilde{H}$.

This consideration implies that $\widetilde{X}+\widetilde{Y}+Z_1$ is a $\delta$-vector
on $\widetilde{G}/\widetilde{H}$. Therefore, we have for $Z_1$ one of the possibilities in Lemma \ref{lem3}.
Besides this, it is easy to see that $[Z_2,\widetilde{X}]=0$, $[Z_3,\widetilde{Y}]=[Z_2,\widetilde{Y}]=0$.
From Proposition \ref{tam} we see that
$$
[Z_1+Z_2+Z_3,\widetilde{X}]=\frac{x_2-x_1}{x_2}[\widetilde{Y},\widetilde{X}]=
[Z_1,\widetilde{X}],
$$
therefore, $[Z_3, \widetilde{X}]=0$. As it is easy to check, this implies $Z_3=0$ if
$\widetilde{X}\neq 0$.

We have the following two possibilities: $\widetilde{c}= 0$ or $\widetilde{c} \neq 0$.

In the first case we have $\widetilde{X}=0$. Since
$\widetilde{W}_{\mathfrak{h}} \in sp(l-1)$ by the case 1) in Lemma \ref{lem3}, and $\widetilde{Y}$
commutes with $sp(l-1)$, we can move $\widetilde{W}_{\mathfrak{h}}$ by some
$\Ad(b)$, $b \in Sp(l-1)$, to a given Cartan subalgebra of $sp(l-1)$, not changing
$\widetilde{Y}$. The vectors ${\bf i}G_q$, $2\leq q \leq l$ generate
such subalgebra. Then we have the item 1) of Lemma.

If $\widetilde{c} \neq 0$, then $\widetilde{X}\neq 0$,  $Z_3=0$ (see above),
$Z_2 \in sp(l-2)$. Since $\widetilde{W}_{\mathfrak{p}}$ commutes with $sp(l-2)$ we can
move $Z_2$ by some $\Ad(b)$, $b \in Sp(l-2)$, to
a given Cartan subalgebra of $sp(l-2)$. The vectors ${\bf i}G_q$, $3\leq q \leq l$ generate such subalgebra.
Thus we get 2) or 3) in Lemma depending on whether $\widetilde{d} =0$ or not.
\end{proof}

Later on we shall need the embedding $\pi:Sp(l)\rightarrow SU(2l)$, which is defined by
$$
A+{\bf j}B \rightarrow
\left(
\begin{array}{rr}
A&-\overline{B}\\
B&\overline{A}\\
\end{array}
\right)
$$
and the corresponding embedding $d\pi :sp(l) \rightarrow su(2l)$, acted by
$$
X+{\bf j}Y \rightarrow
\left(
\begin{array}{rr}
X&-\overline{Y}\\
Y&\overline{X}\\
\end{array}
\right).
$$
It is easy to check the following formulas:
$$
d\pi(E_{ij})=E_{ij}+E_{l+i,l+j},\quad
d\pi({\bf i} F_{ij})= {\bf i} F_{i,j} -{\bf i} F_{l+i,l+j},
$$
$$
d\pi({\bf j} F_{ij})= E_{l+i,j} -E_{i,l+j}, \quad
d\pi({\bf k} F_{ij})= -{\bf i} F_{l+i,j} -{\bf i} F_{i,l+j},
$$
$$
d\pi({\bf i}G_i)={\bf i}G_i-{\bf i}G_{l+i}, \quad
d\pi({\bf j}G_i)=-\sqrt{2}E_{i,l+i}, \quad
d\pi({\bf k}G_i)=-\sqrt{2}{\bf i} F_{i,l+i}.
$$

For any $W\in sp(l)$ we denote by $\Pol(W)$ the characteristic polynomial of the matrix
$d\pi(W)$.
It easy to get the following

\begin{prop}\label{char}
1) If $\widetilde{W}_1=\widetilde{d} {\bf j}G_1  +\sum\limits_{q=2}^l {\alpha}_q {\bf i} G_q$,
then
$$
\Pol(\widetilde{W}_1)=({\lambda}^2+2\widetilde{d}^2)\cdot
\prod\limits_{q=2}^l({\lambda}^2+2{\alpha}_q^2);
$$

2) If $\widetilde{W}_2=\widetilde{c} E_{12}  +
\alpha ({\bf i} G_1 +{\bf i} G_2)+\sum\limits_{q=3}^l {\alpha}_q {\bf i} G_q$,
then
$$
\Pol(\widetilde{W}_2)=\left({\lambda}^4+2(\widetilde{c}^2+2{\alpha}^2){\lambda}^2+
(\widetilde{c}^2-2{\alpha}^2)^2\right)
\cdot
\prod\limits_{q=3}^l({\lambda}^2+2{\alpha}_q^2);
$$

3) If $\widetilde{W}_3=
\widetilde{c}E_{12}+\widetilde{d}{\bf j}G_1-\frac{x_2-x_1}{x_1}\widetilde{d}{\bf j} G_2
+\sum\limits_{q=3}^l {\alpha}_q {\bf i} G_q$,
then
$$
\Pol(\widetilde{W}_3)=\left({\lambda}^4+2\left(\widetilde{c}^2+\widetilde{d}^2+\widetilde{d}^2
\left(\frac{x_2-x_1}{x_1}\right)^2\right){\lambda}^2 +
\left( \widetilde{c}^2+2\widetilde{d}^2\frac{x_2-x_1}{x_1}\right)^2\right)\cdot
\prod\limits_{q=3}^l({\lambda}^2+2{\alpha}_q^2).
$$
\end{prop}

\begin{prop}\label{main} If $x_1<x_2<2x_1$, then
for arbitrary positive $c$ and $d,$ the vector
$$
W=X+Y+Z=cE_{12}+d{\bf j}G_1-\frac{x_2-x_1}{x_1}d{\bf j} G_2
$$
is a $\delta$-vector on $(G/H, \mu=\mu_{x_1,x_2})$.
\end{prop}

\begin{proof}
Suppose that the vector $W=cE_{12}+d{\bf j}G_1-\frac{x_2-x_1}{x_1}d{\bf j} G_2$
is not a $\delta$-vector. Then according to Proposition \ref{osn3}
there is a vector $\widetilde{W}$ in the $\Ad(G)$-orbit of $W$, which has
one of the following forms:

1) $\widetilde{W}_1=\widetilde{d} {\bf j}G_1  +\sum\limits_{q=2}^l {\alpha}_q {\bf i} G_q$,
where $x_2\widetilde{d}^2 > x_2 d^2 +x_1 c^2$;

2) $\widetilde{W}_2=\widetilde{c} E_{12}  +
\alpha ({\bf i} G_1 +{\bf i} G_2)+\sum\limits_{q=3}^l {\alpha}_q {\bf i} G_q$,
where $x_1\widetilde{c}^2 > x_2 d^2 +x_1 c^2$;

3) $\widetilde{W}_3=
\widetilde{c}E_{12}+\widetilde{d}{\bf j}G_1-\frac{x_2-x_1}{x_1}\widetilde{d}{\bf j} G_2
+\sum\limits_{q=3}^l {\alpha}_q {\bf i} G_q$,
where $x_2\widetilde{d}^2+x_1\widetilde{c}^2 > x_2 d^2 +x_1 c^2$.

Note, that for the vector $W$ and a suitable vector $\widetilde{W}_i$ we have
$\Pol(W)=\Pol(\widetilde{W}_i)$, since these vector are in one and the same orbit of the group $\Ad(G)$.
Note, that
$$
\Pol(W)=\left[{\lambda}^4+2\left({c}^2+{d}^2+{d}^2\left(\frac{x_2-x_1}{x_1}\right)^2\right)
{\lambda}^2 +\left({c}^2+2{d}^2\frac{x_2-x_1}{x_1}\right)^2\right]\lambda^{2(l-2)}.
$$
Consider the above three cases separately.

1) In this case we apply the item 1) in Proposition \ref{char}.
Since $\Pol(W)=\Pol(\widetilde{W}_1)$, we see that there is exactly one ${\alpha}_q\neq 0$, and we have
$$
{c}^2+{d}^2+{d}^2\left(\frac{x_2-x_1}{x_1}\right)^2=\widetilde{d}^2+{\alpha}_q^2, \quad
{c}^2+2{d}^2\frac{x_2-x_1}{x_1}=2\widetilde{d}|{\alpha}_q|.
$$
It is easy to see that
$$
(\widetilde{d}-|{\alpha}_q|)^2=d^2\left(\frac{2x_1-x_2}{x_1}\right)^2, \quad
(\widetilde{d}+|{\alpha}_q|)^2=2c^2+d^2\left(\frac{x_2}{x_1}\right)^2.
$$
Therefore,
\begin{equation}
\label{wi}
2\widetilde{d}\leq |\widetilde{d}-|{\alpha}_q||+|\widetilde{d}+|{\alpha}_q||\leq
\sqrt{2c^2+d^2\left(\frac{x_2}{x_1}\right)^2}+d \frac{2x_1-x_2}{x_1}.
\end{equation}
One can easily check that for every real numbers $c, d, x_1, x_2$ with the properties
$c\neq 0, 2x_1>x_2>0$ the following inequality (see Lemma 8 in \cite{BerNik}) is true:
\begin{equation}
\label{le}
\left(|d|(2x_1-x_2)+\sqrt{c^2x_1^2+d^2x_2^2}\right)^2x_2< 2x_1^2(x_1c^2+2x_2d^2).
\end{equation}
Using the inequalities (\ref{wi}) and (\ref{le}), we get that
$$
4\widetilde{d}^2x_2\leq
\left( \sqrt{2c^2+d^2\left(\frac{x_2}{x_1}\right)^2}+d \frac{2x_1-x_2}{x_1} \right)^2x_2<
4(x_1c^2+x_2d^2),
$$
which contradicts to the inequality $x_2\widetilde{d}^2 > x_2 d^2 +x_1 c^2$.

2) In this case we apply the item 2) in Proposition \ref{char}.
We get ${\alpha}_q=0$ for all $q\geq 3$, and
$$
{c}^2+{d}^2+{d}^2\left(\frac{x_2-x_1}{x_1}\right)^2=\widetilde{c}^2+2{\alpha}^2, \quad
{c}^2+2{d}^2\frac{x_2-x_1}{x_1}=|\widetilde{c}^2-2{\alpha}^2|.
$$
Since $0<\frac{x_2-x_1}{x_1}<1$,
then $x_2>x_1\left(1+\left( \frac{x_2-x_1}{x_1}\right)^2\right)$, and
$$
x_1\widetilde{c}^2\leq x_1(\widetilde{c}^2+2{\alpha}^2)
= x_1
\left(c^2+{d}^2\left(1+\left(\frac{x_2-x_1}{x_1}\right)^2\right)\right)<
x_1c^2+x_2d^2,
$$
which contradicts to the inequality $x_1\widetilde{c}^2 > x_2 d^2 +x_1 c^2$.

3) In this case we apply the item 3) in Proposition \ref{char}.
It is easy to see that ${\alpha}_q=0$ for $q\geq 3$, and
$$
{c}^2+{d}^2+{d}^2\left(\frac{x_2-x_1}{x_1}\right)^2=\widetilde{c}^2+\widetilde{d}^2+\widetilde{d}^2
\left(\frac{x_2-x_1}{x_1}\right)^2, \quad
{c}^2+2{d}^2\frac{x_2-x_1}{x_1}=\widetilde{c}^2+2\widetilde{d}^2\frac{x_2-x_1}{x_1},
$$
which easily implies $d=\widetilde{d}$ and $c=\widetilde{c}$. The latter contradicts
to the inequality
$x_2\widetilde{d}^2+x_1\widetilde{c}^2 > x_2 d^2 +x_1 c^2$.
The proposition is proved.
\end{proof}

\begin{prop}\label{vvspom5}
If $x_1\leq x_2 \leq 2x_1$, then the Riemannian manifold
$(G/H=Sp(l)/U(1)\cdot Sp(l-1),\mu={\mu}_{x_1,x_2})$ is
$Sp(l)$-$\delta$-homogeneous.
\end{prop}

\begin{proof}
If $x_1=x_2$, then the metric $\mu$ is $Sp(l)$-normal and,
therefore, it is $Sp(l)$-$\delta$-homogeneous. If we suppose that
$x_2 \in (x_1,2x_1)$, then the proof follows from Proposition \ref{osn1} and
from Proposition \ref{main}. The statement for $x_2=2x_1$ it is easy to get
through a limiting process.
\end{proof}

\begin{proof}[Proof of Theorem \ref{main1}]
If the Riemannian manifold $(Sp(l)/U(1)\cdot
Sp(l-1),\mu={\mu}_{x_1,x_2})$ is $\delta$-homogeneous, then by
Proposition 28 of \cite{BerNik} we get $x_1\leq x_2 \leq 2x_1$. On the
other hand, for $x_2=x_1$ and for $x_2=2x_1$ the metric $\mu$ is
$Sp(l)$-normal homogeneous and $SU(2l)$-normal homogeneous
respectively. From Proposition
\ref{vvspom5} we get, that the Riemannian manifold
$(Sp(l)/U(1)\cdot Sp(l-1),\mu={\mu}_{x_1,x_2})$ is
$\delta$-homogeneous for $2x_1 > x_2 > x_1$. The theorem is proved.
\end{proof}

\begin{remark}
The Riemannian manifolds $(Sp(l)/U(1)\cdot Sp(l-1),\mu={\mu}_{x_1,x_2}), l\geq 2,$ have positive
sectional curvatures and their (exact) pinching constant is $\varepsilon=(\frac{x_2}{4x_1})^2$ if
$0< x_2\leq 2x_1.$ For all other values of $x_1, x_2$ this statement is not true and sectional curvature
is not necessarily nonnegative \cite{Vol}.
\end{remark}


\begin{thebibliography}{99}


\bibitem{Al}
D.V.~Alekseevskii,
{\sl Compact quaternion spaces,} Funk. Anal. Pril., 2(2) (1968) 11--20.

\bibitem{AA}
D.V.~Alekseevsky, A.~Arvanitoyeorgos,
{\sl Riemannian flag manifolds with homogeneous geodesics,} Trans. Amer. Math. Soc., 359 (2007) 3769--3789.

\bibitem{AV}
D.N.~Akhiezer, E.B.~Vinberg,
{\sl Weakly symmetric spaces and spherical varieties,} Transf. Groups, 4 (1999) 3--24.

\bibitem{Ber}
V.N.~Berestovskii,
{\sl Homogeneous manifolds with an intrinsic metric I,} Sib. Mat. Zh. 29(6) (1988)
17-29 (in Russian), English translation in: Siber. Math. J. 29(6)
(1988) 887--897.


\bibitem{BerG}
V.N.~Berestovskii, L.~Guijarro,
{\sl A metric characterization of Riemannian submersions,}
Ann. Global Anal. Geom. 18(6) (2000) 577--588.



\bibitem{BerNik}
V.N.~Berestovskii, Yu.G.~Nikonorov,
{\sl On $\delta$-homogeneous Riemannian manifolds,}
Differential Geometry and its Applications 26(5) (2008) 514--535.

\bibitem{BerNik1}
V.N.~Berestovskii, Yu.G.~Nikonorov,
{\sl On $\delta$-homogeneous Riemannian manifolds,}
Dokl. Akademii Nauk, 415(6) 727-729 (in Russian), English translation in:
Doklady Mathematics 76(1) (2007) 596--598.

\bibitem{BerNik2}
V.N. Berestovskii, Yu.G. Nikonorov,
{\sl The Chebyshev norm on the Lie algebra of the motion group of a compact
homogeneous Finsler manifold}, Contemporary Mathematics and its Applications, V. 60, Algebra, 2008,
99--122 (in Russian).
English translation in:
Journal of Mathematical Sciences, 2009 (to appear).

\bibitem{BerNik3}
V.N.~Berestovskii, Yu.G.~Nikonorov,
{\sl On $\delta$-homogeneous Riemannian manifolds, II,}
Siber. Math. J., 50(2) (2009), 214--222.


\bibitem{BerNik4}
V.N.~Berestovskii, Yu.G.~Nikonorov,
{\sl On Clifford-Wolf Homogeneous Riemannian manifolds,}
Dokl. Akademii Nauk, 423(1) (2008) 7--10 (in Russian), Emglish translation in
Doklady Mathematics 78(3) (2008) 807--810.


\bibitem{BerNik5}
V.N.~Berestovskii, Yu.G.~Nikonorov,
{\sl Clifford-Wolf Homogeneous Riemannian manifolds,}
J. Differ. Geom. (to appear).

\bibitem{BerNik6}
V.N.~Berestovskii, Yu.G.~Nikonorov,
{\sl Killing vector fields of constant length on locally symmetric Riemannian
manifolds,} Transform. Groups, 13(1) (2008), 25--45.


\bibitem{BerNik7}
{\sl V.N.~Berestovskii, Yu.G.~Nikonorov},
Killing vector fields of constant length on Riemannian manifolds, Siber. Math. J., 49(3) (2008), 395--407.

\bibitem{BerNik8}
{\sl V.N.~Berestovskii, Yu.G.~Nikonorov},
Regular and quasiregular isometric flows on Riemannian
manifolds, Siber. Adv. Math., 18(3) (2008), 153--162.


\bibitem{BerPl}
V.N.~Berestovskii, C.~Plaut,
{\sl Homogenous spaces of curvature bounded below,}
J. Geom. Anal. 9(2) (1999) 203--219.

\bibitem{Berg}
M.~Berger,
{\sl Les varietes riemanniennes homogenes normales a courbure strictement positive,}
Ann. Sc. Norm. Super. Pisa, Cl. Sci., IV Ser. 15 (1961) 179--246.

\bibitem{BKV}
J.~Berndt, O.~Kowalski, L.~Vanhecke,
{\sl Geodesics in weakly symmetric spaces,}
Ann. Global Anal. Geom. 15 (1997) 153--156.

\bibitem{Bes}
A.L.~Besse,
{\sl Einstein Manifolds}. Springer-Verlag, Berlin,
Heidelberg, New York, London, Paris, Tokyo, 1987.

\bibitem{Ca}
\'E.~Cartan,
{\sl Sur une classe remarquable d'espaces de Riemann,} Bull. Soc. Math. de France, 54 (1926) 214--264; 55 (1927) 114--134.

\bibitem{Du}
Z.~Du\v{s}ek,
{\sl Structure of geodesics in the flag manifold $SO(7)/U(3)$},
Proceedings of the 10th International Conference on Differential geometry and its application, Olomouc,
Czech Republic, 27--31 August 2007, World Scientific Publishing, 2008, 89--98.

\bibitem{FIP}
M.~Falcitelli, S.~Ianus, A.M.~Pastore,
{\sl Riemannian submersions and related topics,}
World Scientific, New Jersey, ..., 2004.

\bibitem{Hel}
S.~Helgason,
{\sl Differential geometry and symmetric spaces,}
Academic Press Inc., New-York, 1962.


\bibitem{Ker}
 M.~Kerr,
{\sl Some new homogeneous Einstein metrics on symmetric
spaces,} Trans. Amer. Math. Soc. 348 (1996) 153--171.

\bibitem{KN}
S.~Kobayashi, K.~Nomizu,
{\sl Foundations of differential
geometry,} Vol.~I -- A Wiley-Interscience Publication, New York,
1963; Vol.~II -- A Wiley-Interscience Publication, New York, 1969.

\bibitem{Kost}
B.~Kostant,
{\sl On holonomy and homogeneous spaces,} Nagoya Math. J. 12 (1957) 31--54.


\bibitem{KV}
O.~Kowalski, L.~Vanhecke,
{\sl Riemannian manifolds with homogeneous geodesics,} Boll. Unione
Mat. Ital. Ser. B. 5(1) (1991) 189–-246.

\bibitem{La1}
J.~Lauret,
{\sl Commutative spaces which are not wekly symmetric,} Bull. London
Math. Soc. 30 (1998) 29--37.

\bibitem{Mao}
Y.~Mao,
{\sl A converse of the Gelfand theorem,} Proc. of Amer. Math. Soc. 125(9) (1997) 2699--2702.

\bibitem{On}
A.L.~Onishchik,
{\sl Topology of Transitive Transformation
Groups,} Johann Ambrosius Barth: Leipzig, Berlin, Heidelberg, 1994.

\bibitem{Pen}
R.~Penrose,
{\sl Techniques of differential topology in relativity,} SIAM (1972).


\bibitem{S}
A.~Selberg,
{\sl Harmonic Analysis and discontinuous groups in weakly symmetric riemannian spaces,
with applications to Dirichlet series,} J. Indian Math. Soc. 20 (1956) 47--87.

\bibitem{Ta}
H.~Tamaru,
{\sl Riemannin g.o. spaces fibered over irreducible symmetric spaces,} Osaka J. Math. 15 (1998) 55--67.

\bibitem{Tam}
H.~Tamaru,
{\sl Riemannin geodesic orbit spacemetrics on fiber bundles,} Algebra Groups Geom. 36 (1999) 835--851.

\bibitem{Vol}
D.E.~Vol'per,
{\sl Sectional curvatures of nonstandard metrics on $\mathbb{C}P^{2n+1}$,}
Sib. Mat. Zh. 40(1) (1999) 49--56 (in Russian), English translation in: Sib. Math. J.
40(1)(1999) 39--45.

\bibitem{W}
J.A.~Wolf,
{\sl Spaces of constant curvature,} Publish or Perish, Inc., Wilmington, Delaware (U.S.A.), 1984.

\bibitem{W1}
J.A.~Wolf,
{\sl Harmonic Analysis on Commutative Spaces,} American Mathematical Society, 2007.

\bibitem{Zi1}
W.~Ziller,
{\sl Homogeneous Einstein Metrics on Spheres and
Projective Spaces,}  Math. Ann. 259 (1982) 351--358.


\bibitem{Zi96}
W.~Ziller,
{\sl Weakly symmetric spaces,} 355--368. In: Progress
in Nonlinear Differential Equations. V.~20. Topics in geometry: in
memory of Joseph D'Atri. Birkh{\"a}user, 1996.


\end{thebibliography}
\end{document}